\documentclass{amsart}
\usepackage{amsmath,amsfonts,amsthm,amssymb,color,hyperref}
\usepackage{amsrefs}

\usepackage[T1]{fontenc}
\usepackage{pdfsync}

\newcommand{\beq}{\begin{equation}}
\newcommand{\eeq}{\end{equation}}
\newcommand{\bea}{\begin{eqnarray}}
\newcommand{\eea}{\end{eqnarray}}
\newcommand{\beas}{\begin{eqnarray*}}
\newcommand{\eeas}{\end{eqnarray*}}

\newtheorem{theorem}{Theorem}[section]
\newtheorem{definition}[theorem]{Definition}
\newtheorem{proposition}[theorem]{Proposition}
\newtheorem{corollary}[theorem]{Corollary}

\newtheorem{remark}[theorem]{Remark}
\newtheorem{example}[theorem]{Example}
\newtheorem{examples}[theorem]{Examples}
\newtheorem{foo}[theorem]{Remarks}

\title{The subelliptic heat kernel on the three dimensional solvable Lie groups
}

\author{Fabrice Baudoin}\thanks{The first author was supported in part by NSF Grant DMS 0907326}
\author{Matthew Cecil}

\begin{document}

\maketitle

\begin{abstract}
We study the subelliptic heat kernels of the CR three dimensional solvable Lie groups.  We first classify all left-invariant sub-Riemannian structures on three dimensional solvable Lie groups and obtain representations of these groups.  We give expressions for the heat kernels on these groups and obtain heat semigroup gradient bounds using a new type of curvature-dimension inequality.
\end{abstract}


\section{Introduction}\label{s.1}

The motivation of this work is to study the subelliptic  heal kernel and related functional inequalities on the CR three dimensional solvable Lie groups. This is a natural complement to the papers \cite{BB1,Bo} and \cite{BBBC,Ga} that study the subelliptic  heat kernel on the semisimple and nilpotent CR three dimensional Lie groups respectively. It also complements the papers \cite{BW1, W} that study the subelliptic heat kernel on CR Sasakian model spaces.

Suppose $\mathfrak{g}$ is a Lie algebra and $H\subset \mathfrak{g}$ is a subspace endowed with an inner product $\langle\cdot,\cdot\rangle$.  If $G$ denotes a Lie group with Lie algebra $\mathfrak{g}$, then the left-invariant extension of $(H,\langle\cdot,\cdot\rangle)$ determines a smooth distribution and hence a sub-Riemannian structure on $G$.  Conversely, to any sub-Riemannian manifold $(G,M,g)$, where $G$ is a Lie group and $M$ and $g$ are left-invariant, we can associate the triple $(\mathfrak{g},H,\langle\cdot,\cdot\rangle)$, where $\mathfrak{g}$ is the Lie algebra of $G$, $H=M_e$, and $\langle \cdot,\cdot\rangle=g_e$, where $e$ denotes the identity of $G$.

In this work, we will be concerned with left-invariant sub-Riemannian structures and the subelliptic heat kernel on 3-dimensional solvable Lie groups.   It will be assumed throughout that $G$ is a 3-dimensional solvable Lie group with Lie algebra $\mathfrak{g}$ and $H\subset\mathfrak{g}$ is a 2-dimensional subspace endowed with an inner product $\langle\cdot,\cdot\rangle$ which satisfies H\"{o}rmander's condition: $H$ generates $\mathfrak{g}$ under iterated Lie brackets.  We will refer to $H$ as the \textit{horizontal subspace} of $\mathfrak{g}$.  Except for a brief foray into the general case found in Section \ref{s.4.4}, any triple $(\mathfrak{g},H,\langle \cdot,\cdot\rangle)$ discussed will be assumed to be of this form.  We will also use throughout the identification of elements of the Lie algebra $\mathfrak{g}$ with left-invariant differential operators on $G$.

As we show in Section \ref{s.2.1},  to any triple $(\mathfrak{g},H,\langle\cdot,\cdot\rangle)$ we can associate a basis in which the Lie algebra relations take the form
\[
[X,Y]=Z\qquad [X,Z]=\alpha Y+\beta Z\qquad [Y,Z]=0.
\]
where $\alpha \in \mathbb{R}$ and $\beta  \ge 0$ are two parameters. In the underlying canonical  CR structure, the Reeb vector field $R$ is given by
\[R=-\beta Y+Z.\]
 We show that the parameter $\alpha$ is a torsion parameter while  $\beta$ is a curvature parameter. The CR sub-Laplacian is then  the left invariant and subelliptic diffusion operator
\[
 L=X^2+Y^2-\beta X.
 \]
It is symmetric with respect to a left-invariant Haar measure $\mu$ on $G$. The study of functional inequalities related to the Dirichlet form
 \[
 \mathcal{E}(f,g)=-\int_G fLg ~d\mu=\int_G (Xf Xg +YfYg) ~d\mu
\]
is of special interest, because no general techniques are currently known to handle this type of subelliptic operators. The major hurdle to bypass is that, for $\alpha \neq 0$, the CR structure is not of Sasakian type and thus the techniques introduced in  \cite{BB2}  or  \cite{BG1} do not apply.  However, we show here that $L$ satisfies the following generalized curvature dimension inequality: For every $f\in C^\infty$ and $\nu>0$,
\begin{equation*}
\Gamma_2(f)+\nu\Gamma_2^R(f)\geq \frac{1}{2}(Lf)^2+\frac{1}{2}(1-\nu^2\alpha^2)\Gamma^R(f)+(-\alpha^+-\beta^2-\frac{1}{\nu})\Gamma(f),
\end{equation*}
where $\alpha^+=\max{\{\alpha,0\}}$,
\[\Gamma(f,g)=(Xf)(Xg)+(Yf)(Yg), \Gamma^R(f,g)=(Rf)(Rg),\]
\[\Gamma_2(f)=\frac{1}{2}L\Gamma(f)-\Gamma(f,Lf), \Gamma_2^R(f)=\frac{1}{2}L\Gamma^R(f)-\Gamma^R(f,Lf).\]
along with the convention $\Gamma(f)=\Gamma(f,f)$.  This new type of curvature dimension inequality, that also has been shown to be satisfied on more general contact manifolds in \cite{BW2}, opens the door for the study of functional inequalities related to $L$. In particular, we are able to deduce only from it gradient bounds for the heat semigroup $P_t$: If $T>0$ is small enough, then the following estimate holds
\begin{equation*}
\Gamma(P_Tf)+\frac{1}{|\alpha|}\Gamma^R(P_Tf)\leq \frac{\kappa e^{2\kappa T}}{\kappa+|\alpha|(1-e^{2\kappa T})}P_T(\Gamma(f))+\frac{1}{|\alpha|}P_T(\Gamma^R(f)),
\end{equation*}
where $\kappa=\beta^2+\alpha^+$.  We are also able to prove a reverse Poincar\'{e}-type inequality.

\

Another important aspect of our work is the study of explicit formulas for the integral heat kernel of $P_t$. In particular by working with suitably chosen faithful representations of the Lie group, we are able  to establish connections with some exponential functionals of the Brownian motion and therefore to deduce useful probabilistic representations of the heat kernel. We also describe a technique for obtaining spectral representations of the heat kernel and carry this out in one particular case.

\section{sub-Riemannian Structures on $3$-dimensional Solvable Lie Groups}\label{s.2}

The isomorphism classes of 3-dimensional Lie algebras are well documented (see, for example, \cite{FH,J}).  In \cite{AB}, Agrachev and Barilari have provided a classification of sub-Riemannian structures on 3-dimensional Lie groups in terms of two differential invariants, $\chi$ and $\kappa$ (see also \cite{VG} for a related discussion).  We will find it more convenient to classify sub-Riemannian structures on 3-dimensional solvable Lie groups using two different parameters $\alpha$ and $\beta$ which are closely tied to the algebraic structure of the Lie algebra and, as we show in Section \ref{s.2.2}, the geometry of the Tanaka-Webster connection on $G$.

In Section \ref{s.2.1}, we will show that to any triple $(\mathfrak{g},H,\langle\cdot,\cdot\rangle)$ we can associate two real parameters $\alpha$ and $\beta\geq 0$.  Any two triples with the same parameters are related by an Lie algebra isomorphism which acts on as an isometry between horizontal subspaces.  Furthermore, simply scaling the metric on the horizontal subspace changes these parameters in a readily identifiable way.

In Section \ref{s.2.2}, we elaborate on the sub-Riemannian geometry of a Lie group with fixed parameters $(\alpha,\beta)$.  In particular, we exhibit the Reeb vector field and canonical sub-Laplacian.  

\subsection{Classification}\label{s.2.1}

As always, let $(\mathfrak{g},H,\langle\cdot,\cdot\rangle)$ denote triple where $\mathfrak{g}$ is a 3-dimensional solvable Lie algebra, $H\subset \mathfrak{g}$ is a 2-dimensional horizontal subspace satisfying H\"{o}rmander's condition, and $\langle \cdot,\cdot\rangle$ is an inner product on $H$.  Let $\mathfrak{g}'=[\mathfrak{g},\mathfrak{g}]$ denote the derived subalgebra of $\mathfrak{g}$.  Following \cite{FH}, we will refer to the dimension of $\mathfrak{g}'$ as the \textit{rank} of the Lie bracket (viewed as a linear map $[\cdot,\cdot]:\bigwedge^2\mathfrak{g}\rightarrow \mathfrak{g}$), or just simply the rank of $\mathfrak{g}$.  The rank of $\mathfrak{g}$ is either $1$ or $2$; rank $0$ implies that $\mathfrak{g}$ is commutative and contradicts the existence of $H$, while rank $3$ contradicts the solvability of $\mathfrak{g}$.

The following theorem provides the parameters $\alpha$ and $\beta$ we use to classify triples $(\mathfrak{g},H,\langle\cdot,\cdot\rangle)$.
\begin{proposition}\label{p.2.1}
Given a triple $(\mathfrak{g},H,\langle\cdot,\cdot\rangle)$, there exist orthonormal vectors $X,Y\in H$ and a vector $Z\in \mathfrak{g}'$ such that
\begin{equation}
[X,Y]=Z\qquad [X,Z]=\alpha Y+\beta Z\qquad [Y,Z]=0\label{e.2.1}
\end{equation}
for some real $\alpha$ and $\beta\geq 0$.
\end{proposition}
\begin{proof}
We divide the proof into cases according to the rank of $\mathfrak{g}$.  If the rank of $\mathfrak{g}$ is equal to 1, then first observe that $\mathfrak{g}'\cap H=\{0\}$, since otherwise we would have $[H,H]\subset H$. Pick a nonzero element $Z_0\in\mathfrak{g}'$.  In this case, the kernel of the Lie bracket $\ker{[\cdot,\cdot]}$ (viewed as a linear map $[\cdot,\cdot]:\bigwedge^2\mathfrak{g}\rightarrow \mathfrak{g}$) is 2 dimensional.  Define another 2 dimensional subspace $V_H\subset \bigwedge^2\mathfrak{g}$ to be the image of the linear map $\iota:H\rightarrow \bigwedge^2\mathfrak{g}$ defined by
\[\iota(\xi)=\xi\otimes Z_0-Z_0\otimes \xi.\]
The intersection of $V_H$ and $\ker{[\cdot,\cdot]}$ is nontrivial by purely dimensional considerations, so pick a unit vector $Y\in H$ such that $\iota(Y)\in \ker{[\cdot,\cdot]}$.  Pick another unit vector $X_0\in H$ which is orthogonal to $Y$.  Then $[X_0,Y]=\lambda Z_0$ for some nonzero $\lambda$.  We now set $Z_1=\lambda Z_0$ and we have the commutation relations
\[[X_0,Y]=Z_1\qquad [X_0,Z_1]=\beta Z_1\qquad [Y,Z_1]=0\]
for some real $\beta$.  If $\beta\geq 0$, then set $X=X_0$ and $Z=Z_1$.  If $\beta<0$, then set $X=-X_0$ and $Z=-Z_1$.  Either way, we get the commutation relations of (\ref{e.2.1}) with $\alpha=0$ and $\beta\geq 0$.

Now suppose the rank of $\mathfrak{g}$ is 2.  Let's write $\mathfrak{g}'$ as the span of two vectors $\{\xi_1,\xi_2\}$ and pick $\xi_3\in\mathfrak{g}\setminus\mathfrak{g}'$.  Now since $\xi_1\in\mathfrak{g}'=[\mathfrak{g},\mathfrak{g}]$, it can be written the commutator of two elements of $\mathfrak{g}$ and hence the adjoint operator $\operatorname{ad}_{\xi_1}:\mathfrak{g}\rightarrow\mathfrak{g}$ has trace zero.  Since $\operatorname{ad}_{\xi_1}\xi_1=0$ and $\operatorname{ad}_{\xi_1}\xi_3\in\mathfrak{g}'$, it follows that $\operatorname{ad}_{\xi_1}\xi_2=0$.  So when the rank of $\mathfrak{g}$ is 2, $\mathfrak{g}'$ is abelian.  This implies that $H$ and $\mathfrak{g}'$ cannot be equal, since $\mathfrak{g}'$ cannot generate $\mathfrak{g}$ under iterated Lie brackets.  Pick any unit vector $Y\in\mathfrak{g}'\cap H$, then choose another unit vector $X_0\in H$ which is orthogonal to $Y$.  Note that $[X_0,Y]\not\in \operatorname{span}\{Y\}$ since otherwise $H=\operatorname{span}\{X_0,Y\}$ would only generate $H$ under iterated Lie brackets.  Set $Z_0=[X_0,Y]\in \mathfrak{g}'$. Then $\operatorname{span}\{Y,Z_0\}=\mathfrak{g}'$, and since $[X_0,Z_0]\in\mathfrak{g}'$, it follows that $[X_0,Z_0]=\alpha Y+\beta Z$ for some $\alpha,\beta$ with $\alpha\neq 0$.   If $\beta\geq 0$, then set $X=X_0$ and $Z=Z_0$.  Otherwise, if $\beta<0$, then set $X=-X_0$ and $Z=-Z_0$.  We get the commutation relations of (\ref{e.2.1}) with $\alpha\neq 0$ and $\beta\geq 0$.
\end{proof}
The proposition above motivates the following definition.
\begin{definition}
A basis $\{X,Y,Z\}$ for $\mathfrak{g}$ is called a \textit{canonical basis for} $(\mathfrak{g},H,\langle\cdot,\cdot\rangle)$ \textit{with parameters } $\alpha$ \textit{and} $\beta\geq 0$ if $\{X,Y\}$ forms an orthonormal basis for $H$ and the following commutation relations hold
\[[X,Y]=Z\qquad [X,Z]=\alpha Y+\beta Z\qquad [Y,Z]=0.\]
\end{definition}

It is not clear from the statement or proof of Proposition \ref{p.2.1} that canonical bases (or the related parameters, for that matter) are unique for a given triple.  We prove that this is in fact the case in Proposition \ref{p.3390}.  We first specify our criteria for comparing triples.
\begin{definition}
The triples $(\mathfrak{g},H,\langle\cdot,\cdot\rangle)$ and $(\hat{\mathfrak{g}},\hat{H},\widehat{\langle\cdot,\cdot\rangle})$ are \emph{isomorphic} if there exists a Lie algebra isomorphism $\phi:\mathfrak{g}\rightarrow\hat{\mathfrak{g}}$ such that $\phi(H)=\hat{H}$ and $\phi|_H:H\rightarrow\hat{H}$ is an isometry.  The triples are \emph{almost isomorphic} if $\phi|_H:H\rightarrow\hat{H}$ is an isometry after a rescaling of the metric on $\hat{H}$.
\end{definition}

\begin{proposition}\label{p.3390}
Suppose $\{X,Y,Z\}$ is a canonical basis for $(\mathfrak{g},H,\langle\cdot,\cdot\rangle)$ with parameters $(\alpha,\beta)$ and $\{\hat{X},\hat{Y},\hat{Z}\}$ is canonical basis for $(\hat{\mathfrak{g}},\hat{H},\widehat{\langle\cdot,\cdot\rangle})$ with parameters $(\hat{\alpha},\hat{\beta})$.  Then $(\mathfrak{g},H,\langle\cdot,\cdot\rangle)$ and $(\hat{\mathfrak{g}},\hat{H},\widehat{\langle\cdot,\cdot\rangle})$ are almost isomorphic iff there exists a positive constant $C$ such that
\begin{equation}\label{e.2.2}
C\hat{\beta}=\beta\quad\text{and}\quad C^2\hat{\alpha}=\alpha.
\end{equation}
When Eq. (\ref{e.2.2}) holds, then $\phi|_{H}$ is an isometry after a rescaling of the metric on $\hat{H}$ by the constant $C^{-2}$.  In particular, $(\mathfrak{g},H,\langle\cdot,\cdot\rangle)$ and $(\hat{\mathfrak{g}},\hat{H},\widehat{\langle\cdot,\cdot\rangle})$ are isomorphic iff $\alpha=\hat{\alpha}$ and $\beta=\hat{\beta}$; that is, the parameters specify equivalence of triples up to isomorphism.
\end{proposition}
\begin{proof}
Suppose Eq. (\ref{e.2.2}) holds for some $C>0$.  Consider the linear map $\phi:\mathfrak{g}\rightarrow\hat{\mathfrak{g}}$ determined by $\phi(X)=C\hat{X}$, $\phi(Y)=C\hat{Y}$, and $\phi(Z)=C^2\hat{Z}$.  Then $\phi$ is a Lie algebra homomorphism with $\phi(H)=\hat{H}$.
In this case, $\phi|_H$ is an isometry with respect to the scaled metric $C^{-2}\widehat{\langle\cdot,\cdot\rangle}$ on $\hat{H}$.

Now suppose there exists an isomorphism $\phi:\mathfrak{g}\rightarrow\hat{\mathfrak{g}}$ with $\phi(H)=\hat{H}$.   We first observe that the ranks of $\mathfrak{g}$ and $\hat{\mathfrak{g}}$ must be equal, since the derived subalgebra is a subalgebra.  In particular, this says that either $\alpha=\hat{\alpha}=0$ (rank 1 case) or $\alpha\neq 0$ and $\hat{\alpha}\neq 0$ (rank 2 case).

If $\alpha=\hat{\alpha}=0$, then we can set $C=\frac{\beta}{\hat{\beta}}$ provided $\hat{\beta}\neq 0$.  When $\alpha=\hat{\alpha}=0$ and $\hat{\beta}=0$, the center of $\hat{\mathfrak{g}}$, which is another subalgebra, has dimension $1$, and the fact that $\mathfrak{g}$ and $\hat{\mathfrak{g}}$ are isomorphic forces the dimension of the center of $\mathfrak{g}$ to be 1 also, i.e. $\beta=0$ as well.  In this case, any $C>0$ suffices for Eq. (\ref{e.2.2}).

If both $\alpha\neq 0$ and $\hat{\alpha}\neq 0$, then the rank of both $\mathfrak{g}$ and $\hat{\mathfrak{g}}$ is 2.  In this case, the adjoint representations are faithful.  Let $\pi:\mathfrak{g}\rightarrow M_3(\mathbb{C})$ and $\hat{\pi}:\hat{\mathfrak{g}}\rightarrow M_3(\mathbb{C})$ denote the adjoint representations written with respect to the bases $\{X,Y,Z\}$ and $\{\hat{X},\hat{Y},\hat{Z}\}$.  Note that the matrix representation of the transformation $\operatorname{ad}_{\phi(X)}$ in the basis $\{\phi(X),\phi(Y),\phi(Z)\}$ is equal to $\pi(X)$; in particular, $\pi(X)$ and $\hat{\pi}(\phi(X))$ have the same characteristic polynomial.
The characteristic polynomial of $\pi(X)$ is
\begin{equation}\label{e.2.3}
p_{\pi(X)}(\lambda)=-\lambda^3+\beta\lambda^2+\alpha\lambda.
\end{equation}
Now $\phi$ sends $X$ to an element of the form $C\hat{X}+a\hat{Y}+b\hat{Z}\in\hat{\mathfrak{g}}$, for some non-zero constant $C$ and some constants $a,b$.  $C$ cannot be zero, otherwise $\phi$ would take $X$ into $\hat{\mathfrak{g}}'$.  One can compute that
\begin{equation}\label{e.2.4}
p_{\hat{\pi}(C\hat{X}+a\hat{Y}+b\hat{Z})}(\lambda)=-\lambda^3+C\hat{\beta}\lambda^2+C^2\hat{\alpha}\lambda,
\end{equation}
which, since we have an equality of lines (\ref{e.2.3}) and (\ref{e.2.4}), implies the relations of ($\ref{e.2.2}$).
\end{proof}

We have therefore established that any triple $(\mathfrak{g},H,\langle\cdot, \cdot\rangle)$ is described uniquely (up to isomorphism) by the two parameters $\alpha$ and $\beta\geq 0$.  If $(\mathfrak{g},H,\langle\cdot,\cdot\rangle)$ and $(\hat{\mathfrak{g}},\hat{H},\widehat{\langle\cdot,\cdot\rangle})$ are triples with different parameters, then it may well be that $\mathfrak{g}$ is isomorphic as a Lie algebra to $\hat{\mathfrak{g}}$ even if the relations in Eq. (\ref{e.2.2}) are not satisfied.  However, in general, the isomorphism will not map $H$ to $\hat{H}$.  
\begin{remark}
One can show that the parameters differential invariants $\chi$ and $\kappa$ of \cite{AB} are related to $\alpha$ and $\beta$ by the formula
\[\chi=\frac{|\alpha|}{2}\qquad \kappa=-\beta^2-\frac{\alpha}{2}.\]
As shown in the next section, these constants have a natural geometric interpretation.
\end{remark}

\subsection{CR structure of the triple  $(\mathfrak{g},H,\langle\cdot,\cdot\rangle)$}\label{s.2.2}

We now give the geometric interpretation of the parameters $\alpha$ and $\beta$. We show that the parameters $\alpha$ and $\beta$ are related to the torsion and curvature of the Tanaka-Webster connection respectively.

Let $(\mathfrak{g},H,\langle\cdot,\cdot\rangle)$ denote a triple with parameters $(\alpha,\beta)$.  As described in the previous section, this means that $\mathfrak{g}=\operatorname{span}\{X,Y,Z\}$, $\{X,Y\}$ is an orthonormal basis for $H$, and
\[[X,Y]=Z\qquad [X,Z]=\alpha Y+\beta Z\qquad [Y,Z]=0.\]

The Lie group $G$ carries then a natural left invariant CR structure which is given by the complex subbundle generated by $X+iY$ (for details about CR manifolds, we refer to the book \cite{CR}). The left invariant 1-form $\theta=dZ$ is a contact form on $G$. Straightforward computations show that
\[
\mathcal{L}_X\theta =-dY,\quad \mathcal{L}_Y \theta =dX, \quad \mathcal{L}_Z \theta=\beta dX.
\]

As a consequence, the Reeb vector field of the contact form $\theta$ is
\[R=-\beta Y+Z.\]
We note, then, that we have the following commutation relations
\begin{equation}\label{e.2.5}
[X,Y]=\beta Y+R\qquad [X,R]=\alpha Y\qquad [Y,R]=0.
\end{equation}
Note that $Y$ and $R$ span $\mathfrak{g}'$.

These commutation relations make easy to compute the Christoffel's symbols of the Tanaka-Webster connection $\nabla$:
\[
\nabla_X X=0, \quad\nabla_X Y =\frac{1}{2} \beta Y ,\quad \nabla_X R=0,
\]
\[
\nabla_Y X =-\beta Y,\quad \nabla_Y Y=0,\quad \nabla_Y R=0,
\]
\[
\nabla_R X=-\frac{1}{2} \alpha Y,\quad \nabla_R Y=\frac{1}{2} \alpha X,\quad \nabla_R R=0.
\]
In particular, one computes that the pseudo Hermitian torsion $\tau$ of $\nabla$ is the horizontal endomorphism characterized by
\[
\tau(X) =\frac{1}{2} \alpha Y,\quad  \tau(Y)=\frac{1}{2} \alpha X.
\]
As a conclusion, $\alpha$ is the torsion parameter of the canonical CR structure on  the triple $(\mathfrak{g},H,\langle\cdot,\cdot\rangle)$.  The parameter $\beta$ turns out to be a curvature parameter. Indeed, if $\mathbf{Ric}$ denotes the Ricci curvature tensor of the connection $\nabla$, then one computes that for $V \in H$,
\[
\mathbf{Ric}(V,V)=-\left(\beta^2+\frac{\alpha}{2}\right)  \| V \|^2.
\]

If we consider the left-invariant Haar measure determined by the volume form $\theta\wedge d\theta= dX\wedge dY\wedge dR$, then we have
\[X^*=-X+\beta\qquad Y^*=-Y\]
It follows that the CR   sub-Laplacian is the left-invariant subelliptic operator
\begin{equation}\label{e.1.2.2}
L=-X^*X-Y^*Y=X^2+Y^2-\beta X.
\end{equation}
\section{Examples and Representations}\label{s.3}

In this section, we give real representations of 3-dimensional solvable Lie algebras determined by the parameters $(\alpha,\beta)$ and their associated solvable Lie groups.  In many ways, this analysis follows the typical arguments made in classifying three dimensional Lie algebras (see \cite{FH,J} for example).  However, our analysis differs in that the Lie algebra elements $X$ and $Y$ we exhibit cannot be scaled independently due to the requirement that they form an orthonormal basis for $H$; we view the triple $(\mathfrak{g},H,\langle\cdot,\cdot\rangle)$ as given and look for representations which are best adapted to this structure.

Recall from Section \ref{s.2.2} that all our Lie algebras can be written as $\mathfrak{g}=\operatorname{span}\{X,Y,R\}$ and satisfy
\begin{equation}\label{e.2.9}
[X,Y]=\beta Y+R\qquad [X,R]=\alpha Y\qquad [Y,R]=0.
\end{equation}
When $\alpha\neq 0$, or equivalently the rank of $\mathfrak{g}$ is 2, $\mathfrak{g}$ is centerless and hence the adjoint representation is faithful.  In this case, the adjoint representation of $\mathfrak{g}$ in the basis $\{Y,R,X\}$ is given by
\begin{equation}\label{e.5544} X=\left(\begin{array}{ccc} \beta & \alpha & 0\\ 1 & 0 & 0\\ 0 & 0 & 0\end{array}\right)\qquad Y=\left(\begin{array}{ccc} 0 & 0 & -\beta\\ 0 & 0 & -1\\ 0 & 0 & 0\end{array}\right)\qquad R=\left(\begin{array}{ccc} 0 & 0 & -\alpha\\ 0 & 0 & 0\\ 0 & 0 & 0\end{array}\right)\end{equation}
If $\{\xi_1,\xi_2\}$ is another basis for $\mathfrak{g}'$, then the adjoint representation in the basis $\{\xi_1,\xi_2,X\}$ will take the partitioned form
\begin{equation}
X=\left(\begin{array}{cc} A & \overline{0}\\\overline{0}^T & 0\end{array}\right)\qquad Y=\left(\begin{array}{cc} \mathbf{0} & \overline{y}\\ \overline{0}^T & 0\end{array}\right)\qquad R=\left(\begin{array}{cc} \mathbf{0} & \overline{r}\\ \overline{0}^T & 0\end{array}\right)\label{e.2.10}
\end{equation}
where $A$ denotes the transformation $\operatorname{ad}_X:\mathfrak{g}'\rightarrow\mathfrak{g}'$ in the basis $\{\xi_1,\xi_2\}$, $\overline{y}$ and $\overline{r}$ are some vectors in $\mathbb{R}^2$ which depend on the basis $\{\xi_1,\xi_2\}$, and
\[\mathbf{0}=\left(\begin{array}{cc} 0 & 0\\ 0 & 0\end{array}\right)\qquad \overline{0}=\left(\begin{array}{c} 0\\ 0\end{array}\right).\]
Thus, we have a family of faithful representations of $\mathfrak{g}$ determined by the choice of basis $\{\xi_1,\xi_2\}$.  For a given $(\alpha,\beta)$, certain choices of bases will give rise to more `natural' coordinates when exponentiated.  Note also that if $Y$ and $R$ are scaled by the same nonzero number, the Lie algebra relations are maintained.  This scaling will also be used to simplify the resulting adjoint representations.  We will see that even in the rank 1 case, we have a faithful representation of $\mathfrak{g}$ of the form of Eq. (\ref{e.2.10}).

It is easy to check that exponentials of matrices of the form of Eq. (\ref{e.2.10}) are given by
\[\exp{(sX)}=\left(\begin{array}{cc} \exp{(sA)} & \overline{0}\\\overline{0}^T & 1\end{array}\right)\quad \exp{(sY)}=\left(\begin{array}{cc} I& s\overline{y}\\ \overline{0}^T & 1\end{array}\right)\quad \exp{(sR)}=\left(\begin{array}{cc} I & s\overline{r}\\ \overline{0}^T & 1\end{array}\right).\]
In which case, we have natural coordinates $(\theta, x,y)$ on the the Lie group $G$ associated to the Lie algebra representation in Eq. (\ref{e.2.10}).  These coordinates are given by
\begin{equation}
G=\left\{\left.\left(\begin{array}{cc} \exp(\theta A) & \overline{x}\\ \overline{0}^T & 1\end{array}\right)\right|\overline{x}=(x,y)^T\text{ with }\theta ,x,y\in\mathbb{R}\right\}.\label{e.2.11}
\end{equation}
Note that Eq. (\ref{e.5544}) implies that $G$ is unimodular iif $\beta=0$.  In these coordinates on the group $G$, elements of $\mathfrak{g}$, considered as left-invariant vector fields on $G$, are given by the differential operators
\begin{equation}\tilde{X}=\frac{\partial}{\partial \theta}\qquad
\tilde{Y}=\overline{\nabla} \left(\exp{(\theta A)}\overline{y}\right)\qquad
\tilde{R}=\overline{\nabla} \left(\exp{(\theta A)}\overline{r}\right)\label{e.2.12}\end{equation}
where $\overline{\nabla}$ is the row vector valued operator
\begin{equation}\label{e.2.15}
\overline{\nabla}=(\frac{\partial}{\partial x},\frac{\partial}{\partial y}).
\end{equation}
Note that the coefficients of $\frac{\partial}{\partial x}$ and $\frac{\partial}{\partial y}$ in the vector fields given in Eq. (\ref{e.2.12}) only depend on $\theta$.  In the sections that follow, these computations are carried out for the various regimes of the parameters $(\alpha,\beta)$.

\subsection{Rank 1 $(\alpha=0)$}\label{s.3.1}
When $\alpha=0$, $\mathfrak{g}$ as given by Eq. (\ref{e.2.9}) has a nontrivial center and so the adjoint representation of $\mathfrak{g}$ is not faithful.  We nonetheless have representations of the form of Eq. (\ref{e.2.10}).

\subsubsection{$\beta=0$}\label{s.3.1.1}
When $\beta=0$, $\mathfrak{g}$ is the well-known Heisenberg Lie algebra
\[[X,Y]=R\qquad [X,R]=0\qquad [Y,R]=0.\]
The Heisenberg Lie algebra has the representation
\[X=\left(\begin{array}{ccc} 0 & 1 & 0\\ 0 & 0 & 0\\ 0 & 0 & 0\end{array}\right)\qquad Y=\left(\begin{array}{ccc} 0 & 0 & 0\\ 0 & 0 & 1\\ 0 & 0 & 0\end{array}\right)\qquad R=\left(\begin{array}{ccc} 0 & 0 & 1\\ 0 & 0 & 0\\ 0 & 0 & 0\end{array}\right),\]
while the Heisenberg Lie group is the group of $3\times 3$ upper triangular matrices
\[G=\left\{\left.\left(\begin{array}{ccc} 1 & \theta & x\\ 0 & 1 & y\\ 0 & 0 & 1\end{array}\right)\right|\theta,x,y\in\mathbb{R}\right\}.\]
In the coordinates $(\theta,x,y)$, elements of $\mathfrak{g}$ correspond to the left-invariant differential operators
\[X=\frac{\partial}{\partial \theta}\qquad Y=\theta\frac{\partial}{\partial x}+\frac{\partial}{\partial y}\qquad R=\frac{\partial}{\partial x}\]
and the sub-Laplacian takes the form
\[L=\frac{\partial^2}{\partial \theta^2}+\left(\theta \frac{\partial}{\partial x}+\frac{\partial}{\partial y}\right)^2.\]

\subsubsection{$\beta\neq 0$}\label{s.3.1.2}
When $\beta\neq 0$, we have the relations
\[[X,Y]=\beta Y+R\qquad [X,R]=0\qquad [Y,R]=0.\]
We have a faithful representation for $\mathfrak{g}$ given by
\[X=\left(\begin{array}{ccc} \beta & 0 & 0\\ 0 & 0 & 0\\ 0 & 0 & 0\end{array}\right)\qquad Y=\left(\begin{array}{ccc} 0 & 0 & 1\\ 0 & 0 & 1\\ 0 & 0 & 0\end{array}\right)\qquad R=\left(\begin{array}{ccc} 0 & 0 & 0\\ 0 & 0 & -\beta\\ 0 & 0 & 0\end{array}\right),\]
which is the Lie algebra of the simply-connected group
\[G=\left\{\left.\left(\begin{array}{ccc} e^{\beta\theta} & 0 & x\\ 0 & 1 & y\\ 0 & 0 & 1\end{array}\right)\right|\theta,x,y\in\mathbb{R}\right\}.\]
In the coordinates $(\theta,x,y)$, elements of $\mathfrak{g}$ correspond to the left-invariant differential operators
\[X=\frac{\partial}{\partial \theta}\qquad Y=e^{\beta\theta}\frac{\partial}{\partial x}+\frac{\partial}{\partial y}\qquad R=-\beta\frac{\partial}{\partial y}\]
and the sub-Laplacian
\[L=\frac{\partial^2}{\partial \theta^2}-\beta\frac{\partial}{\partial \theta}+\left(e^{\beta\theta}\frac{\partial}{\partial x}+\frac{\partial}{\partial y}\right)^2.\]

When $\beta=1$, if we set $\xi_1=Y+R,\xi_2=-X$, and $\xi_3=-R$, then we have the relations
\[[\xi_1,\xi_2]=\xi_1\qquad[\xi_1,\xi_3]=0\qquad [\xi_2,\xi_3]=0\]
which is the Lie algebra of $A^+(\mathbb{R})\oplus\mathbb{R}$, the affine maps on the plane which act as orientation preserving on one axis and translations on the other axis.

\subsection{Rank 2 $(\alpha\neq 0)$}
When $\alpha\neq 0$, the operator $\operatorname{ad}_X|_{\mathfrak{g}'}$ is invertible, hence we arrive at representations of these Lie algebras based on the (real) Jordan normal form of $\operatorname{ad}_X|_{\mathfrak{g}'}$.  Note that the characteristic polynomial of $\operatorname{ad}_X|_{\mathfrak{g}'}$ is
\[p(\lambda)=\left|\begin{array}{cc} \beta-\lambda & \alpha\\1 & -\lambda\end{array}\right|=\lambda^2-\beta\lambda-\alpha.\]
We consider the following cases based on the discriminant of $p$, $\delta=\beta^2+4\alpha$.

\subsubsection{$\delta=\beta^2+4\alpha>0$}
In this case, $\operatorname{ad}_X|_{\mathfrak{g}'}$ has two distinct real eigenvalues, $\lambda_1=\frac{1}{2}(\beta + \sqrt{\delta})$ and $\lambda_2=\frac{1}{2}(\beta - \sqrt{\delta})$.  The basis $\{\xi_1,\xi_2\}$, where $\xi_1=\lambda_1 Y+R$ and $\xi_2=\lambda_2 Y+R$, diagonalizes $\operatorname{ad}_X|_{\mathfrak{g}'}$.  If we compute the adjoint representation of $\mathfrak{g}$ in the basis $\{\xi_1,\xi_2,X\}$ and then scale $Y$ and $R$ by $\sqrt{\delta}$ we obtain
\[X=\left(\begin{array}{ccc} \lambda_1 & 0 & 0\\ 0 & \lambda_2 & 0\\ 0 & 0 & 0\end{array}\right)\qquad Y=\left(\begin{array}{ccc} 0 & 0 & -\lambda_1\\ 0 & 0 & \lambda_2\\ 0 & 0 & 0\end{array}\right)\qquad R=\left(\begin{array}{ccc} 0 & 0 & -\alpha\\ 0 & 0 & \alpha\\ 0 & 0 & 0\end{array}\right),\]
which is the Lie algebra of the simply-connected matrix group
\[G=\left\{\left.\left(\begin{array}{ccc} e^{\lambda_1\theta} & 0 & x\\ 0 & e^{\lambda_2\theta} & y\\ 0 & 0 & 1\end{array}\right)\right|\theta,x,y\in\mathbb{R}\right\}.\]
In the coordinates $(\theta,x,y)$, elements of $\mathfrak{g}$ correspond to the left-invariant differential operators
\[X=\frac{\partial}{\partial \theta}\qquad Y=-\lambda_1 e^{\lambda_1\theta}\frac{\partial}{\partial x}+\lambda_2 e^{\lambda_2\theta}\frac{\partial}{\partial y}\qquad R=-\alpha e^{\lambda_1\theta}\frac{\partial}{\partial x}+\alpha e^{\lambda_2\theta}\frac{\partial}{\partial y}\]
and the sub-Laplacian
\[L=\frac{\partial^2}{\partial \theta^2}-\beta\frac{\partial}{\partial\theta}+\left(-\lambda_1 e^{\lambda_1\theta}\frac{\partial}{\partial x}+\lambda_2 e^{\lambda_2\theta}\frac{\partial}{\partial y}\right)^2.\]

When $\alpha=1$ and $\beta=0$, we have $\lambda_1=1,$ and $\lambda_2=-1$.  In this case, we have there relations
\[[X,Y]=R\qquad [X,R]=Y\qquad [Y,R]=0.\]
This is the Lie algebra $\mathfrak{solv}^-$.

\subsubsection{$\delta=\beta^2+4\alpha<0$}\label{s.4499}
In this case, $\operatorname{ad}_X|_{\mathfrak{g}'}$ has two complex conjugate eigenvalues, $\frac{1}{2}(\beta +i\sqrt{-\delta})$ and $\frac{1}{2}(\beta -i\sqrt{-\delta})$.  For simplicity, set $\rho=\frac{\beta}{2}$, $\omega=\frac{\sqrt{-\delta}}{2}$, $\xi_1=\rho Y+R$ and $\xi_2=-\omega Y$.  If we compute the adjoint representation of $\mathfrak{g}$ in the basis $\{\xi_1,\xi_2,X\}$ and then scale $Y$ and $R$ by $\omega$ we obtain
\[X=\left(\begin{array}{ccc} \rho & -\omega & 0\\ \omega & \rho & 0\\ 0 & 0 & 0\end{array}\right)\qquad Y=\left(\begin{array}{ccc} 0 & 0 & -\omega\\ 0 & 0 & \rho\\ 0 & 0 & 0\end{array}\right)\qquad R=\left(\begin{array}{ccc} 0 & 0 & 0\\ 0 & 0 & -(\rho^2+\omega^2)\\ 0 & 0 & 0\end{array}\right),\]
which is the Lie algebra of the matrix group
\[G=\left\{\left.\left(\begin{array}{ccc} e^{\rho\theta}\cos{(\omega\theta)} & -e^{\rho\theta}\sin{(\omega\theta)} & x\\ e^{\rho\theta}\sin{(\omega\theta)} & e^{\rho\theta}\cos{(\omega\theta)} & y\\ 0 & 0 & 1\end{array}\right)\right|\theta,x,y\in\mathbb{R}\right\}.\]
In the coordinates $(\theta,x,y)$, elements of $\mathfrak{g}$ correspond to the left-invariant differential operators
\[X=\frac{\partial}{\partial \theta}\qquad R=(\rho^2+\omega^2)e^{\rho\theta}\sin{(\omega\theta)}\frac{\partial}{\partial x}-(\rho^2+\omega^2)e^{\rho\theta}\cos{(\omega\theta)}\frac{\partial}{\partial y}\]
and
\begin{align*}
Y&=\left(-\omega e^{\rho\theta}\cos{(\omega\theta)}-\rho e^{\rho\theta}\sin{(\omega\theta)}\right)\frac{\partial}{\partial x}+\left(-\omega e^{\rho\theta}\sin{(\omega\theta)}+\rho e^{\rho\theta}\cos{(\omega\theta)}\right)\frac{\partial}{\partial y},\end{align*}
as well as the sub-Laplacian
\begin{align*}
L&=\frac{\partial^2}{\partial\theta^2}-\beta\frac{\partial}{\partial\theta}\\&+e^{\beta\theta}\left(\left(-\omega \cos{(\omega\theta)}-\rho \sin{(\omega\theta)}\right)\frac{\partial}{\partial x}+\left(-\omega \sin{(\omega\theta)}+\rho \cos{(\omega\theta)}\right)\frac{\partial}{\partial y}\right)^2.
\end{align*}
Note that $\rho^2+\omega^2=-\alpha$.

When $\alpha=-1$ and $\beta=0$, we have $\delta=-4$, $\rho=0$, and $\omega=1$.  In this case, we have there relations
\[[X,Y]=R\qquad [X,R]=-Y\qquad [Y,R]=0.\]
This is the Lie algebra $\mathfrak{se}(2)$, the Lie algebra of the Euclidean motions of the plane.  Note that for these parameters
\[L=\frac{\partial^2}{\partial\theta^2}+\left(\cos{(\theta)}\frac{\partial}{\partial x}+\sin{(\theta)}\frac{\partial}{\partial y}\right)^2.\]
This case is elaborated on in Section \ref{s.4.2} where we use the above form of $L$ to obtain a representation of the heat kernel using the Fourier transform.

\subsubsection{$\delta=\beta^2+4\alpha=0$}
In this case, $\operatorname{ad}_X|_{\mathfrak{g}'}$ has one real real eigenvalue $\lambda=\frac{\beta}{2}$ but is not diagonalizable.  
If we set $\xi_1=\lambda Y+R$ and $\xi_2=(1+\lambda)Y+R$, then the adjoint representation of $\mathfrak{g}$ in the basis $\{\xi_1,\xi_2,X\}$ can be written as
\[X=\left(\begin{array}{ccc} \lambda & 1 & 0\\ 0 & \lambda & 0\\ 0 & 0 & 0\end{array}\right)\qquad Y=\left(\begin{array}{ccc} 0 & 0 & \lambda-1\\ 0 & 0 & -\lambda\\ 0 & 0 & 0\end{array}\right)\qquad R=\left(\begin{array}{ccc} 0 & 0 & -\lambda^2\\ 0 & 0 & \lambda^2\\ 0 & 0 & 0\end{array}\right)\]
which is the Lie algebra of the matrix group
\[G=\left\{\left.\left(\begin{array}{ccc} e^{\lambda\theta} & \theta e^{\lambda\theta} & x\\ 0 & e^{\lambda\theta} & y\\ 0 & 0 & 1\end{array}\right)\right|\theta,x,y\in\mathbb{R}\right\}.\]
In the coordinates $(\theta,x,y)$, elements of $\mathfrak{g}$ correspond to the left-invariant differential operators
\[X=\frac{\partial}{\partial \theta}\qquad Y=e^{\lambda\theta}\left(\lambda-1-\lambda\theta\right)\frac{\partial}{\partial x}-\lambda e^{\lambda\theta}\frac{\partial}{\partial y}\]
and
\[R=\lambda^2e^{\lambda\theta}\left(\theta-1\right)\frac{\partial}{\partial x}+\lambda^2 e^{\lambda\theta}\frac{\partial}{\partial y},\]
as well as the sub-Laplacian
\[L=\frac{\partial^2}{\partial\theta^2}-\beta\frac{\partial}{\partial\theta}+e^{\beta\theta}\left(\left(\lambda-1-\lambda\theta\right)\frac{\partial}{\partial x}-\lambda\frac{\partial}{\partial y}\right)^2.\]

\section{The Subelliptic Heat Kernel and Heat Semigroup}\label{s.4}

Let $(\mathfrak{g},H,\langle\cdot,\cdot\rangle)$ denote a triple with parameters $(\alpha,\beta)$ and $G$ a 3-dimensional solvable Lie group with Lie algebra $\mathfrak{g}$.  Let $L$ denote the second-order left-invariant elliptic operator
\[L=X^2+Y^2-\beta X.\]
As described in section \ref{s.2.2}, $L$ is the canonical sub-Laplacian on $G$.  $L$ is a symmetric non-positive operator with respect to the left-invariant Haar measure $\mu$ determined by the volume form $dX\wedge dY\wedge dR$.  We will let $P_t$ denote the semigroup generated by $L$ and $p_t$ the corresponding integral heat kernel so that, for any $g\in G$,
\[(P_tf)(g)=\int_G f(h)p_t(h^{-1}g)~d\mu(h).\]
In the following subsection, we give probabilistic representations of the heat kernel using the coordinates arising form the representations given in Section \ref{s.3}.  We then briefly comment on analytic representations of the Fourier transform of the heat kernel.  

\subsection{Probabilistic Representations of $p_t$}\label{s.4.1}
As described in Section \ref{s.3}, using the coordinates $(\theta ,x,y)$ on the group $G$ described in Eq. (\ref{e.2.11}), one can identify elements of $\mathfrak{g}$ with the left-invariant operators
\begin{equation*}X=\frac{\partial}{\partial \theta}\qquad
Y=\overline{\nabla} \left(\exp{(\theta A)}\overline{y}\right)\qquad
R=\overline{\nabla}\left(\exp{(\theta A)}\overline{r}\right)\end{equation*}

For a fixed $g=(\theta,x,y)\in G$, let $Z^g(\cdot)$ denote the $G$-valued process which solves
\begin{equation}dZ^g(t)=X(Z^g(t))dB(t)+Y(Z^g(t))dW(t)-\beta X(Z^g(t))dt\qquad Z^g(0)=g\label{e.340}
\end{equation}
where $B(\cdot)$ and $W(\cdot)$ are independent standard Brownian motions.  Then $Z^g$ has generator $L$ and hence the density of the endpoint distribution of $Z^g$ is equal to the heat kernel:
\[\mathbb{P}(Z^g(t)\in A)=\int_A p_t(h^{-1}g)~d\mu(h).\]
Since $Z^g$ is left-invariant, it suffices to examine the heat kernel for one particular choice of $g$.  When $g=e=(0,0,0)$, it is not difficult to see that, in the coordinates $(\theta,x,y)$, $Z^e(t)=Z(t)$ takes the form
\begin{equation}
Z(t)
=\left(B(t)-\beta t,\int_0^t \exp{((B(s)-\beta s)A)}\overline{y}~dW(s)\right)\label{e.341}
\end{equation}
Note that the second component in Eq. (\ref{e.341}) is a vector; we are considering
\[Z_A(t):=\int_0^t \exp{((B(s)-\beta s)A)}\overline{y}~dW(s)\]
 as an $\mathbb{R}^2$-valued process.  For fixed $s\rightarrow B(s)$, $Z_A(t)$ is a mean $(0,0)$ Gaussian random variable with covariance matrix
\begin{equation}\label{e.4.3}
\Sigma_t=\int_0^t \exp((B(s)-\beta s)A)\overline{y}\overline{y}^T\exp((B(s)-\beta s)A^T)~ds.
\end{equation}
It follows that we can represent the density of the endpoint distribution of $Z(t)$ by
\begin{equation}\label{e.4.4}
p_{t}(\theta,x,y)=\frac{1}{(2\pi )^{3/2}\sqrt{t}}e^{-\frac{\theta^2}{2t}}\mathbb{E}\left[\frac{1}{\sqrt{|\Sigma_t|}}\exp{\left(-\frac{1}{2}\left(\begin{array}{c} x\\y\end{array}\right)^T\Sigma^{-1}_t\left(\begin{array}{c} x\\y\end{array}\right)\right)}\left|\right.  B_{t}=\theta\right]
\end{equation}
where $|\Sigma_t|$ denotes the determinant of $\Sigma_t$.  In the examples below, we compute this representation for the various regimes of $(\alpha,\beta)$ described in Section \ref{s.3}.
\begin{example}[$\alpha=\beta=0$]
In this case, for fixed $s\rightarrow B(s)$, we have covariance matrix
\[\Sigma_t=\left(\begin{array}{cc} \int_0^t B(s)^2~ds & \int_0^t B(s)~ds\\ \int_0^t B(s)~ds & t\end{array}\right),\]
which has determinant
\[|\Sigma_t|=t\int_0^t B(s)^2~ds-\left(\int_0^t B(s)~ds\right)^2.\]
It follows that
\[p_{t}(\theta,x,y)=\frac{1}{(2\pi )^{3/2}\sqrt{t}}e^{-\frac{\theta^2}{2t}}\mathbb{E}\left[\frac{1}{\sqrt{|\Sigma_t|}}\exp{\left(-\frac{1}{2|\Sigma_t|}\int_0^t\left(x-yB(s)\right)^2ds\right)}\left|\right.  B_{t}=\theta\right].\]
\end{example}
\begin{example}[$\alpha=0,\beta\neq 0$]
In this case for fixed $s\rightarrow B(s)$, we have covariance matrix
\[\Sigma_t=\left(\begin{array}{cc} \int_0^t e^{2\beta (B(s)-\beta s)}~ds & \int_0^t e^{\beta (B(s)-\beta s)}~ds\\ \int_0^t e^{\beta (B(s)-\beta s)}~ds & t\end{array}\right),\]
which has determinant
\[|\Sigma_t|=t\int_0^t e^{2\beta (B(s)-\beta s)}~ds-\left(\int_0^t e^{\beta (B(s)-\beta s)}~ds\right)^2.\]
It follows that
\[p_{t}(\theta,x,y)=\frac{1}{(2\pi )^{3/2}\sqrt{t}}e^{-\frac{\theta^2}{2t}}\mathbb{E}\left[\frac{1}{\sqrt{|\Sigma_t|}}\exp{\left(-\frac{1}{2 |\Sigma_t|}\int_0^t\left(x-ye^{\beta (B(s)-\beta s)}\right)^2ds\right)}\left|\right.  B_{t}=\theta\right].\]
\end{example}

\begin{example}[$\alpha\neq 0, \delta=\beta^2+4\alpha>0$]
Recall that we set $\lambda_1=\frac{1}{2}(\beta+\sqrt{\delta})$ and $\lambda_1=\frac{1}{2}(\beta-\sqrt{\delta})$.  In this case for fixed $s\rightarrow B(s)$, we have covariance matrix
\[\Sigma_t=\left(\begin{array}{cc} \lambda_1^2\int_0^t e^{2\lambda_1(B(s)-\beta s)}~ds & -\lambda_1\lambda_2\int_0^t e^{(\lambda_1+\lambda_2) (B(s)-\beta s)}~ds\\ -\lambda_1\lambda_2\int_0^t e^{(\lambda_1+\lambda_2)(B(s)-\beta s)}~ds & \lambda_2^2\int_0^t e^{2\lambda_2(B(s)-\beta s)}~ds\end{array}\right),\]
which has determinant
\[|\Sigma_t|=\lambda_1^2\lambda_2^2\left(\left(\int_0^t e^{2\lambda_1 (B(s)-\beta s)}~ds\right)\left(\int_0^t e^{2\lambda_2 (B(s)-\beta s)}~ds\right)-\left(\int_0^t e^{(\lambda_1+\lambda_2)(B(s)-\beta s)}~ds\right)^2\right).\]
It follows that
\begin{align*}
p_{t}(\theta,x,y)&=\frac{1}{(2\pi )^{3/2}\sqrt{t}}e^{-\frac{\theta^2}{2t}}\\ &\times \mathbb{E}\left[\frac{1}{\sqrt{|\Sigma_t|}}\exp{\left(-\frac{1}{2 |\Sigma_t|}\int_0^t\left(\lambda_2xe^{\lambda_2 (B(s)-\beta s)}+\lambda_1ye^{\lambda_1 (B(s)-\beta s)}\right)^2ds\right)}\left|\right.  B_{t}=\theta\right].
\end{align*}
\end{example}

\begin{example}[$\alpha\neq 0, \delta=\beta^2+4\alpha<0$]
Recall that we set $\rho=\frac{\beta}{2}$ and $\omega=\frac{\sqrt{-\delta}}{2}$ and that $\rho^2+\omega^2=-\alpha$.  In this case for fixed $s\rightarrow B(s)$, we have the covariance matrix with entries
\[(\Sigma_t)_{11}=-\alpha\int_0^t e^{2\rho(B(s)-\beta s)}\cos^2{(\omega(B(s)-\beta s)-\theta_0)}~ds,\]
\[(\Sigma_t)_{21}=(\Sigma_t)_{12}=-\alpha\int_0^t e^{2\rho(B(s)-\beta s)}\sin{(\omega(B(s)-\beta s)-\theta_0)}\cos{(\omega(B(s)-\beta s)-\theta_0)}~ds\]
and
\[(\Sigma_t)_{22}=-\alpha\int_0^t e^{2\rho(B(s)-\beta s)}\sin^2{(\omega(B(s)-\beta s)-\theta_0)}~ds,\]
where $0\leq \theta_0<\frac{\pi}{2}$ is the angle satisfying $\tan{\theta_0}=\frac{\rho}{\omega}$.  The covariance matrix has determinant
\begin{align*}
|\Sigma_t|&=\alpha^2\left(\int_0^t e^{2\rho (B(s)-\beta s)}\sin^2{(\omega(B(s)-\beta s)-\theta_0)}~ds\right)\\&\times \left( \int_0^t e^{2\rho (B(s)-\beta s)}\cos^2{(\omega(B(s)-\beta s)-\theta_0)}~ds\right)\\&-\alpha^2\left( \int_0^t e^{2\rho (B(s)-\beta s)}\sin{(\omega(B(s)-\beta s)-\theta_0)}\cos{(\omega(B(s)-\beta s)-\theta_0)}~ds\right)^2.
\end{align*}
It follows that
\begin{align*}
p_{t}(\theta,x,y)&=\frac{1}{(2\pi )^{3/2}\sqrt{t}}e^{-\frac{\theta^2}{2t}}\mathbb{E}\left[\frac{1}{\sqrt{|\Sigma_t|}}\exp{\left(\frac{\alpha}{2 |\Sigma_t|}I_{\rho,\omega}(B,x,y)\right)}\left|\right.  B_{t}=\theta\right],
\end{align*}
where
\begin{align*}
&I_{\rho,\omega}(B,x,y)\\&=\int_0^t\left(xe^{\rho(B(s)-\beta s)}\sin{(\omega(B(s)-\beta s)-\theta_0)}-ye^{\rho (B(s)-\beta s)}\cos{(\omega(B(s)-\beta s)-\theta_0)}\right)^2ds.
\end{align*}
\end{example}

\begin{example}[$\alpha\neq 0, \delta=\beta^2+4\alpha=0$]
Recall that we set $\lambda=\frac{\beta}{2}$.  In this case for fixed $s\rightarrow B(s)$, we have covariance matrix with entries
\[(\Sigma_t)_{11}=\int_0^t e^{2\lambda(B(s)-\beta s)}(\lambda-1-\lambda(B(s)-\beta s))^2~ds,\]
\[(\Sigma_t)_{21}=(\Sigma_t)_{12}=-\lambda\int_0^t e^{2\lambda(B(s)-\beta s)}(\lambda-1-\lambda(B(s)-\beta s))~ds\]
and
\[(\Sigma_t)_{22}=\lambda^2\int_0^t e^{2\lambda(B(s)-\beta s)}~ds,\]
which has determinant
\begin{align*}
|\Sigma_t|&=\lambda^2\left(\int_0^t e^{2\lambda(B(s)-\beta s)}~ds\right)\left(\int_0^t e^{2\lambda(B(s)-\beta s)}(\lambda-1-\lambda(B(s)-\beta s))^2~ds\right)\\ &-\lambda^2\left(\int_0^t e^{2\lambda(B(s)-\beta s)}(\lambda-1-\lambda(B(s)-\beta s))~ds\right)^2.
\end{align*}
It follows that
\begin{align*}
p_{t}(\theta,x,y)&=\frac{1}{(2\pi )^{3/2}\sqrt{t}}e^{-\frac{\theta^2}{2t}}\mathbb{E}\left[\frac{1}{\sqrt{|\Sigma_t|}}\exp{\left(-\frac{1}{2 |\Sigma_t|}I_\lambda(B,x,y)\right)}\left|\right.  B_{t}=\theta\right],
\end{align*}
where
\[I_\lambda(B,x,y)=\int_0^te^{2\lambda(B(s)-\beta s)}\left(x\lambda +y(\lambda-1-\lambda(B(s)-\beta s))\right)^2~ds.\]
\end{example}

\subsection{Spectral Representations of $p_t$}\label{s.4.2}
In this subsection, we show how the coordinates described in Eq. (\ref{e.2.11}) can be used, along with the Fourier transforms in the variables $x$ and $y$, to arrive at expressions for the heat kernel.  We compute in the simple case $\alpha=-1$ and $\beta=0$ (which corresponds to the group $SE(2)$).  A similar description of the subelliptic heat kernel on $SE(2)$ can be found in Section 4.5 of \cite{ABGR}.

Using the general coordinates described in Eq. (\ref{e.2.11}), we can write
\[L=\frac{\partial^2}{\partial \theta^2}-\beta\frac{\partial}{\partial\theta}+\left(\begin{array}{c}\frac{\partial}{\partial x}\\\frac{\partial}{\partial y}\end{array}\right)^T\exp{(\theta A)}\overline{y}\overline{y}^T\exp{(\theta A^T)}\left(\begin{array}{c}\frac{\partial}{\partial x}\\\frac{\partial}{\partial y}\end{array}\right)\]
If we apply the Fourier transform in the variables $(x,y)$ (sending them to $(\xi_1,\xi_2)$), we get the operator
\[\hat{L}=\frac{\partial^2}{\partial \theta^2}-\beta\frac{\partial}{\partial\theta}-\left(\begin{array}{c}\xi_1\\\xi_2\end{array}\right)^T\exp{(\theta A)}\overline{y}\overline{y}^T\exp{(\theta A^T)}\left(\begin{array}{c}\xi_1\\\xi_2\end{array}\right).\]
Note that we are using as the definition of the Fourier transform
\[\hat{f}(\xi_1,\xi_2)=\int_{\mathbb{R}^2} e^{-ix\xi_1-iy\xi_2}f(x,y)dxdy,\]
for which the inverse Fourier transform is given by
\[\check{f}(x,y)=\frac{1}{4\pi^2}\int_{\mathbb{R}^2}e^{ix\xi_1+iy\xi_2}f(\xi_1,\xi_2)d\xi_1d\xi_2.\]
For the moment ignoring questions of convergence, if one can find eigenfunctions $\{f_i^{(\xi_1,\xi_2)}(\theta)\}$ and eigenvalues $\{\lambda_i^{(\xi_1,\xi_2)}\}$ which will depend on $(\xi_1,\xi_2)$ for $\hat{L}$, then a solution to the heat equation $u(t,\theta,x,y)$ solving
\[\frac{\partial u}{\partial t}=Lu\qquad u(0,\theta,x,y)=\psi(\theta,x,y)\]
can be written as
\begin{equation}\label{e.3990}
u(t,\theta,x,y)=\frac{1}{4\pi^2}\int_{\mathbb{R}^2}e^{ix\xi_1+iy\xi_2}\left(\sum_{i=1}^\infty a_i(\xi_1,\xi_2)e^{t\lambda_i^{(\xi_1,\xi_2)}}f_i^{(\xi_1,\xi_2)}(\theta)\right)d\xi_1d\xi_2
\end{equation}
where
\[\hat{\psi}(\theta,\xi_1,\xi_2)=\sum_{i=1}^\infty a_i(\xi_1,\xi_2)f_i^{(\xi_1,\xi_2)}(\theta).\]
Eq. (\ref{e.3990}) gives an expression for the heat kernel when $\psi(\theta,x,y)=\delta(\theta)\delta(x)\delta(y)$.  In the example below, we carry out these computations in a relatively simple case.

\begin{example}[$\alpha=-1$,$\beta=0$]

As discussed in Section \ref{s.4499}, when $\alpha=-1$ and $\beta=0$, we have the Lie algebra relations
\[[X,Y]=R\qquad [X,R]=-Y\qquad [Y,R]=0,\]
which is the Lie algebra $\mathfrak{se}(2)$, which is the Lie algebra of the Euclidean motions of the plane.  The adjoint representation exponentiates to the group
\[G=\left\{\left.\left(\begin{array}{ccc} \cos{\theta} & -\sin{\theta} & x\\ \sin{\theta} & \cos{\theta} & y\\ 0 & 0 & 1\end{array}\right)\right|\theta,x,y\in\mathbb{R}\right\}.\]
Note that this group is not simply connected.  In the coordinates $(\theta,x,y)$,
\[L=\frac{\partial^2}{\partial\theta^2}+\left(\cos{\theta}\frac{\partial}{\partial x}+\sin{\theta}\frac{\partial}{\partial y}\right)^2.\]
After applying the Fourier transform, we have the operator
\[\hat{L}=\frac{\partial^2}{\partial\theta^2}-\left(\xi_1\cos{\theta}+\xi_2\sin{\theta}\right)^2.\]

Suppose now that  $u(t,\theta,x,y)$ solves
\[\frac{\partial u}{\partial t}=Lu\qquad u(0,\theta,x,y)=\psi(\theta,x,y),\]
then $\hat{u}(t,\theta,\xi_1,\xi_2)$ solves
\[\frac{\partial \hat{u}}{\partial t}=\hat{L}\hat{u}=\frac{\partial^2\hat{u}}{\partial \theta^2}-\left(\xi_1\cos{\theta}+\xi_2\sin{\theta}\right)^2\hat{u}.\]
We switch to polar coordinates $(\xi_1,\xi_2)\rightarrow (\rho\cos{\phi},\rho\sin{\phi})$ to rewrite this as
\begin{align*}
\frac{\partial \hat{u}}{\partial t}&=\frac{\partial^2\hat{u}}{\partial \theta^2}-\rho^2(\cos{\theta}\cos{\phi}+\sin{\theta}\sin{\phi})^2\hat{u}\\ &=\frac{\partial^2\hat{u}}{\partial \theta^2}-\rho^2\cos^2{(\theta-\phi)}\hat{u}\\ &=\frac{\partial^2\hat{u}}{\partial \theta^2}-\left(\frac{\rho^2}{2}+\frac{\rho^2}{2}\cos(2(\theta-\phi))\right)\hat{u}.
\end{align*}
Setting $A\hat{u}:=\frac{\partial^2\hat{u}}{\partial \theta^2}-\left(\frac{\rho^2}{2}+\frac{\rho^2}{2}\cos(2(\theta-\phi))\right)\hat{u}$ we search for eigenfunctions of $A$.  An eigenfunction of $A$ with eigenvalue $\lambda$ will solve
\begin{equation}\label{e.11}
\frac{\partial^2\hat{u}}{\partial \theta^2}-\left(\frac{\rho^2}{2}+\lambda+\frac{\rho^2}{2}\cos(2(\theta-\phi))\right)\hat{u}=0
\end{equation}
Note that if $\hat{w}$ solves
\begin{equation}\label{e.12}
\frac{\partial^2\hat{w}}{\partial \theta^2}-\left(\frac{\rho^2}{2}+\lambda+\frac{\rho^2}{2}\cos(2\theta)\right)\hat{w}=0
\end{equation}
then $\hat{u}(\theta)=\hat{w}(\theta-\phi)$ will solve Eq. (\ref{e.11}).  Eq. (\ref{e.12}) can be rewritten as Mathieu's differential equation
\begin{equation}\label{e.13}
\frac{\partial^2\hat{w}}{\partial \theta^2}+\left(a-2q\cos(2\theta)\right)\hat{w}=0
\end{equation}
where $a=-\frac{\rho^2}{2}-\lambda$ and $q=\frac{\rho^2}{4}$.

Solutions exist to Eq. (\ref{e.13}) for any choice of $(a,q)$.  However, if we consider $q$ as fixed, then Eq. (\ref{e.13}) will have $2\pi$-periodic solutions for only certain values of the parameter $a$, the \textit{characteristic values}, indexed by a non-negative integer.  When $a$ is a characteristic value, it can be shown that any periodic solution is continuous in the parameter $q$ and that there cannot be two linearly independent periodic solutions except in the case where $q=0$ (where solutions are $\cos{(k\theta)}$ and $\sin{(k\theta)}$ with characteristic values $k^2$).  Those $2\pi$-periodic solutions which reduce to $\cos{(k\theta)}$ when $q=0$ are denoted $\textup{ce}_k(\theta,q)$, while those that reduce to $\sin{(k\theta)}$ are denoted $\textup{se}_k(\theta,q)$.  The characteristic value of $\textup{ce}_k(\theta,q)$ is denoted $a_k(q)$, while the characteristic value of $\textup{se}_k(\theta,q)$ is $b_k(q)$.  These functions are traditionally normalized so that
\[\pi=\int_0^{2\pi}\textup{ce}_k(\theta,q)^2d\theta=\int_0^{2\pi}\textup{se}_k(\theta,q)^2d\theta\]
for all values of $q$.  Furthermore, the set $\{\textup{ce}_k(\cdot,q),\textup{se}_k(\cdot,q)\}_{k=0}^\infty$ is orthogonal in $L^2([0,2\pi])$.  It follows that many functions $F$ on $[0,2\pi]$ can be expanded in a Mathieu function series:
\[F(\theta)=\sum_{k=0}^\infty A_k\textup{ce}_k(\theta,q)+B_k\textup{se}_r(\theta,q)\]
where
\[A_k=\frac{1}{\pi}\int_0^{2\pi}F(\theta)\textup{ce}_k(\theta,q)d\theta\qquad B_k=\frac{1}{\pi}\int_0^{2\pi}F(\theta)\textup{se}_k(\theta,q)d\theta\]
with $A_0=0$.  It turns out that $F$ has such a Mathieu expansion provided it has a Fourier expansion.  These facts and more can be found in \cite{ARSC,NIST}.

This gives our eigenfunctions of $A$ as the functions $\textup{ce}_k(\theta-\phi,\frac{\rho^2}{4})$ and $\textup{se}_k(\theta-\phi,\frac{\rho^2}{4})$ with eigenvalues $\alpha_k(\rho):=-\frac{\rho^2}{2}-a_k(\frac{\rho^2}{4})$ and $\beta_k(\rho):=-\frac{\rho^2}{2}-b_k(\frac{\rho^2}{4})$ respectively.  Note that since these functions are $2\pi$ periodic, they have the same orthogonality relations as those not shifted by $\phi$.

If follows that $\hat{u}$ has an expansion in polar coordinates in terms of  Mathieu functions
\[\hat{u}(t,\theta,\rho,\phi)=\sum_{k=0}^\infty A_ke^{\alpha_k(\rho) t}\textup{ce}_k(\theta-\phi,\frac{\rho^2}{4})+B_ke^{\beta_k(\rho) t}\textup{se}_k(\theta-\phi,\frac{\rho^2}{4})\]

The heat kernel $p_t(\theta,x,y)$ has the property that $p_0(\theta,x,y)=\delta(\theta)\delta(x)\delta(y)$ (where $\delta(x)$ denotes the point mass at $x=0$), and so since the Fourier transform of a delta function is the constant function 1, $\hat{p}_0(\theta,\xi_1,\xi_2)=\delta(\theta)$.  From this and the symmetries of the Mathieu functions, we see that
\[A_k=\frac{1}{\pi}\int_0^{2\pi}\delta(\theta)\textup{ce}_k(\theta-\phi,\frac{\rho^2}{4})d\theta=\frac{1}{\pi}\textup{ce}_k(-\phi,\frac{\rho^2}{4})=\frac{1}{\pi}\textup{ce}_k(\phi,\frac{\rho^2}{4})\]
and
\[B_k=\frac{1}{\pi}\int_0^{2\pi}\delta(\theta)\textup{se}_k(\theta-\phi,\frac{\rho^2}{4})d\theta=\frac{1}{\pi}\textup{se}_k(-\phi,\frac{\rho^2}{4})=-\frac{1}{\pi}\textup{se}_k(\phi,\frac{\rho^2}{4}).\]
Therefore, in polar coordinates, the Fourier transform of the heat kernel takes the form
\[\hat{p}_t(\theta,\rho,\phi)=\frac{1}{\pi}\sum_{k=0}^\infty \textup{ce}_k(\phi,\frac{\rho^2}{4})e^{\alpha_k(\rho) t}\textup{ce}_k(\theta-\phi,\frac{\rho^2}{4})-\textup{se}_k(\phi,\frac{\rho^2}{4})e^{\beta_k(\rho) t}\textup{se}_k(\theta-\phi,\frac{\rho^2}{4})\]
The heat kernel therefore takes the following integral form (obtained by applying the inverse Fourier transform in polar coordinates):
\[p_t(\theta,x,y)=\frac{1}{4\pi^2}\int_0^{2\pi}\int_0^\infty e^{ix\rho\cos{\phi}+iy\rho\sin{\phi}}\hat{p}_t(\theta,\rho,\phi)\rho ~d\rho ~d\theta.\]

\end{example}

\section{Heat semigroup gradient bounds}

\subsection{Curvature-Dimension Inequality}\label{s.4.3}

Again, we assume that $\mathfrak{g}$ is a Lie algebra determined by the relations
\[[X,Y]=\beta Y+R\qquad [X,R]=\alpha Y\qquad [Y,R]=0\]
for some $\alpha$ and $\beta\geq 0$, and $G$ is a Lie group with Lie algebra $\mathfrak{g}$.  Recall that the operator $L$ is the left-invariant differential operator on $G$ defined by
\[L=X^2+Y^2-\beta X.\]
Define the carr\'{e} du champs bilinear forms
\[\Gamma(f,g)=(Xf)(Xg)+(Yf)(Yg)\]
and
\[\Gamma^R(f,g)=(Rf)(Rg).\]
We will denote $\Gamma(f):=\Gamma(f,f)$ and $\Gamma^R(f):=\Gamma^R(f,f)$.  We also define
\[\Gamma_2(f)=\frac{1}{2}L\Gamma(f)-\Gamma(f,Lf)\]
and
\[\Gamma_2^R(f)=\frac{1}{2}L\Gamma^R(f)-\Gamma^R(f,Lf).\]

The purpose of this subsection is to prove the following.
\begin{proposition}\label{p.4.3.2}
For every $f\in C^\infty(G)$ and $\nu>0$,
\begin{equation}\label{e.4.3.2}
\Gamma_2(f)+\nu\Gamma_2^R(f)\geq \frac{1}{2}(Lf)^2+\frac{1}{2}(1-\nu^2\alpha^2)\Gamma^R(f)+(-\alpha^+-\beta^2-\frac{1}{\nu})\Gamma(f),
\end{equation}
where $\alpha^+=\max\{\alpha,0\}$.
\end{proposition}
\begin{proof}

The following simple estimate will be useful
\begin{equation}\label{e.256}
(Lf)^2=(X^2f+(Y^2-\beta X)f)^2\leq 2(X^2f)^2+2(Y^2f-\beta Xf)^2.
\end{equation}
Calculation reveals that
\begin{align*}
\Gamma_2(f) & =(X^2f)^2+(Y^2f-\beta Xf)^2+(XYf)^2+(YXf)^2\\&-\beta^2(Xf)^2-2(Xf)(YRf)+2(Yf)(XRf)\\&+\beta(Yf)((XY+YX)f)-(\alpha+\beta^2)(Yf)^2-\beta(Yf)(Rf).
\end{align*}
Now
\begin{align*}
(XYf)^2+(YXf)^2 &=\frac{1}{2}((XY+XY)f)^2+\frac{1}{2}(\beta Yf+Rf)^2
\end{align*}
and so it follows that
\begin{align*}
\Gamma_2(f) &=(X^2f)^2+(Y^2f-\beta Xf)^2+\frac{1}{2}((XY+XY)f+\beta Yf)^2\\&+\frac{1}{2}(Rf)^2-\beta^2(Xf)^2-(\alpha+\beta^2)(Yf)^2\\ &+2(Yf)(XRf)-2(Xf)(YRf).
\end{align*}
We also observe that
\begin{align*}
\Gamma_2^R(f)&=(XRf)^2+(YRf)^2+\alpha(Rf)((XY+YX)f-\beta(Yf)).
\end{align*}

We now compute the expression of interest using the inequality in Eq. (\ref{e.256}):
\begin{align}
\Gamma_2(f)+\nu\Gamma_2^R(f)&\geq \frac{1}{2}(Lf)^2+\frac{1}{2}(Rf)^2-\beta^2(Xf)^2-(\alpha+\beta^2)(Yf)^2\nonumber\\ &-2(Xf)(YRf)+\nu (YRf)^2\label{e.5.2}\\&+2(Yf)(XRf)+ \nu (XRf)^2\label{e.5.3}\\ &+\frac{1}{2}((XY+XY)f+\beta Yf)^2+\nu\alpha(Rf)((XY+YX)f-\beta(Yf)).\label{e.5.4}
\end{align}
We treat these terms line by line by completing the square.  We note that line (\ref{e.5.2}) can be written
\[-2(Xf)(YRf)+\nu (YRf)^2=(\sqrt{\nu}(YRf)-\frac{1}{\sqrt{\nu}}(Xf))^2-\frac{1}{\nu}(Xf)^2\geq -\frac{1}{\nu}(Xf)^2,\]
and similarly line (\ref{e.5.3}) can be written
\[2(Yf)(XRf)+ \nu (XRf)^2=(\sqrt{\nu}(XRf)+\frac{1}{\sqrt{\nu}}(Yf))^2-\frac{1}{\nu}(Yf)^2\geq -\frac{1}{\nu}(Yf)^2.\]
We also complete the square in line (\ref{e.5.4}) to find
\begin{align*}
& \frac{1}{2}((XY+XY)f+\beta Yf)^2+\nu\alpha(Rf)((XY+YX)f-\beta(Yf))\\ &=\frac{1}{2}(((XY+YX)f+\beta Yf)+\nu\alpha(Rf))^2-\frac{1}{2}\nu^2\alpha^2(Rf)^2\\ &\geq -\frac{1}{2}\nu^2\alpha^2(Rf)^2
\end{align*}
It follows that
\begin{align*}
\Gamma_2(f)+\nu\Gamma_2^R(f)&\geq \frac{1}{2}(Lf)^2+(\frac{1}{2}-\frac{1}{2}\nu^2\alpha^2)(Rf)^2-(\beta^2+\frac{1}{\nu})(Xf)^2-(\alpha+\beta^2+\frac{1}{\nu})(Yf)^2\\ &\geq \frac{1}{2}(Lf)^2+\frac{1}{2}(1-\nu^2\alpha^2)\Gamma^R(f)+(-\alpha^+-\beta^2-\frac{1}{\nu})\Gamma(f),
\end{align*}
where $\alpha^+=\max\{\alpha,0\}$.
\end{proof}

\begin{remark}\label{r.4.3.1}
When $\alpha=0$, the above curvature dimension inequality becomes
\[\Gamma_2(f)+\nu\Gamma_2^R(f)\geq \frac{1}{2}(Lf)^2+\frac{1}{2}\Gamma^R(f)+(-\beta^2-\frac{1}{\nu})\Gamma(f),\]
which is a generalized curvature-dimension inequality $CD(-\beta^2,\frac{1}{2},1,2)$ of the type addressed in \cite{BB1} and \cite{BG1}.
\end{remark}

\subsection{Functional Inequalities}\label{s.4.4}

In this section, we use the curvature-dimension inequality of the last section to derive some functional inequalities.  The technique is similar to that found in the motivating works \cite{BBBC,BB2,BG1}, and relies on the use of a parabolic comparison theorem (see Proposition \ref{p.4.4.3} below).  In order to invoke this comparison, we first need to establish that derivatives of the heat kernel applied to compactly supported are bounded on $G$ uniformly in time.  In Section \ref{s.4.4.1}, we prove this result in broad generality.  In Section \ref{s.38}, we use the curvature-dimension inequalities of Section \ref{s.4.3} to derive heat semigroup gradient bounds.

\subsubsection{Preliminaries}\label{s.4.4.1}

In this subsection, we treat the general case.  Let $G$ now denote an arbitrary Lie group with Lie algebra $\mathfrak{g}$.  We continue to identify $X\in\mathfrak{g}$ with its left-invariant extension, and let $\hat{X}$ denote the right-invariant extension of $X$.  Let $H\subset \mathfrak{g}$ denote a horizontal subspace and $\langle\cdot,\cdot\rangle_H$ an inner product on $H$.  The left-invariant extension of $(H,\langle\cdot,\cdot\rangle_H)$ determines a sub-Riemannian structure on $G$.  We will denote the sub-Riemannian distance by $d_H$, and for convenience, we will denote $|g|_H=d_H(e,g)$.  For $g\in G$, let $B(g,r)$ denote the sub-Riemannian ball of radius $r$.

Let $L$ denote a left-invariant diffusion operator if the form
\[L=X_0+\sum_{k=1}^n X_k^2,\]
where $\operatorname{span}\{X_1,...,X_n\}=H$ and $\{X_1,...,X_n\}$ satisfies H\"{o}rmander's bracket generating condition.  We assume that $L$ is symmetric with respect to a left-invariant Haar measure $\mu$.  Let $P_t=e^{tL}$ denote the corresponding heat semigroup and $p_t$ the convolution kernel, i.e.
\[(P_tf)(g)=\int_G f(h)p_t(h^{-1}g)~d\mu(h).\]
The fact that $p_t$ is stochastically complete, i.e. $P_t1=1$ for all $t\geq 1$, follows from the left-invariance of $L$.  More generally, stochastic completeness will follow from a volume doubling condition (see Theorem 5.5.4 of \cite{sc_aspects}) which is satisfied for sub-Riemannian balls of small radius in this setting (see Theorem V.1.1 of \cite{VSC}, for example).  Finally, we recall the following heat kernel bounds of Theorem IX.1.3 of \cite{VSC}: There exists a positive integer $\nu$ and positive constants $C,a$ such that for any $\epsilon\in (0,1)$
\begin{equation}\label{e.4.4.4}
p_t(g)\leq C t^{-\nu/2}e^{-at}e^{- |g|_H^2/(4+\epsilon)t}
\end{equation}
for all $g\in G$ and $t>0$; in addition,
\begin{equation}\label{e.4.4.5}
p_t(g)\geq  C t^{-\nu/2}e^{- C|g|_H^2/t}
\end{equation}
for all $g\in G$ and $t\in (0,1).$


The main purpose of this subsection is to prove the following theorem:
\begin{theorem}\label{t.4.4.1}
Suppose $f\in C_c^\infty(G)$ and $X$ is any left-invariant vector field.  Then for any $T\geq0$,
\[\sup_{0\leq t\leq T}||XP_tf||_\infty<\infty.\]
\end{theorem}

To prove Theorem \ref{t.4.4.1}, we first extend the given left-invariant sub-Riemannian structure on $G$ to a left-invariant Riemannian structure.  The resulting bound given in the proof of Theorem \ref{t.4.4.1} will depend on the choice of extension, but the bound will be finite regardless of extension.  To this end, let $\langle \cdot,\cdot\rangle$ denote an inner product on $\mathfrak{g}$ such that $\langle X,Y\rangle=\langle X,Y\rangle_H$ whenever $X,Y\in H$.  Extend $\langle\cdot,\cdot\rangle$ to a left-invariant Riemannian metric on $G$.  Let $d$ denote the Riemannian distance on $G$ and $|g|=d(e,g)$.  Note that for any $g,h\in G$, $|g|\leq |g|_H$ and $|gh|_H\leq |g|_H+|h|_H$.

For a linear operator $U:\mathfrak{g}\rightarrow\mathfrak{g}$, we let $||U||_{op}$ denote the operator norm of $U$ computed using the inner product on $\langle\cdot,\cdot\rangle$ on $\mathfrak{g}$ described above.  We will use the notation $|X|_\mathfrak{g}=\sqrt{\langle X,X\rangle}.$  For a vector field $Y$ on $G$ (not necessarily invariant), we will let $|Y(g)|$ denote length (with respect to the Riemannian metric) of the vector $Y(g)\in T_gG$.  For $f\in C^\infty(G)$, let $\nabla f$ denote the gradient of $f$.

\begin{proposition}\label{p.4.4.1}
There exists constants $C,c$ such that for all $g\in G$,
\[||\operatorname{Ad}_g||_{op}\leq Ce^{c|g|_H}.\]
\end{proposition}
\begin{proof}
Let $\gamma:[0,1]\rightarrow G$ be a differentiable horizontal path connecting $e$ to $g$.  Define a sequence of times $0=t_0<t_1<t_2<...<t_n=1$ recursively by
\[t_{i+1}=\left\{\begin{array}{cc}\inf_{t_i<s<1}\{s|d_H(\gamma(t_i),\gamma(s))=1\} & \text{if such a }s\text{ exists}\\ 1 & \text{otherwise}\end{array}\right..\]
Note that this implies that $n-1<l(\gamma)$, where $l(\gamma)$ denotes the sub-Riemannian length of $\gamma$. Set
\[C:=\sup_{g\in B(e,1)}||\operatorname{Ad}_g||_{op},\]
where $B(e,1)$ denotes the sub-Riemannian ball about the identity $e$ of radius $1$.  Since
\[g=\gamma(t_0)^{-1}\gamma(t_1)\gamma(t_1)^{-1}\gamma(t_2)\hdots \gamma(t_{n-1})\gamma(t_{n-1})^{-1}\gamma(t_n),\]
and $\operatorname{Ad}:G\rightarrow \operatorname{End}(\mathfrak{g})$ is a homomorphism, it follows that
\[\operatorname{Ad}_g=\operatorname{Ad}_{\gamma(t_0)\gamma(t_1)^{-1}}\operatorname{Ad}_{\gamma(t_1)\gamma(t_2)^{-1}}\hdots \operatorname{Ad}_{\gamma(t_{n-1})^{-1}\gamma(t_n)},\]
and so
\[||\operatorname{Ad}_g||_{op}\leq C^n\leq CC^{l(\gamma)}.\]
The result follows after setting $c=\log{C}$ and taking the infimum over all such paths $\gamma$.
\end{proof}
\begin{corollary}\label{c.1}
For any $X\in\mathfrak{g}$ and $g\in G$
\[|\hat{X}(g)|\leq  Ce^{c|g|_H}|X|_\mathfrak{g},\]
and consequently, for any $g,h\in G$,
\[|\widehat{Ad_gX}(h)|\leq C^2e^{c(|g|_H+|h|_H)}|X|_\mathfrak{g}.\]
\end{corollary}
\begin{proof}
Note that
\[\hat{X}(g)=\frac{d}{ds}|_{s=0} \exp{(sX)}g=\frac{d}{ds}|_{s=0} gg^{-1}\exp{(sX)}g=\frac{d}{ds}|_{s=0}g\exp{(s\operatorname{Ad}_{g^{-1}}X)},\]
and so
\[|\hat{X}(g)|=|\operatorname{Ad}_{g^{-1}}X|_\mathfrak{g}\leq Ce^{c|g|_H}|X|_\mathfrak{g}.\]
Finally,
\[|\widehat{Ad_gX}(h)|\leq Ce^{c|h|_H}|Ad_gX|_\mathfrak{g}\leq C^2e^{c|h|_H}e^{c|g|_H}|X|_\mathfrak{g}.\]
\end{proof}

\begin{proof}[Proof of Theorem \ref{t.4.4.1}]
We first observe that at time $t=0$, $||XP_tf||_\infty=||Xf||_\infty<\infty$.  So it suffices to show that
\[\sup_{0< t\leq T}||XP_tf||_\infty<\infty.\]
Now since for any $F\in C^\infty(G)$,
\[(XF)(g)=\frac{d}{ds}|_{s=0}F(ge^{sX})=\frac{d}{ds}|_{s=0}F(ge^{sX}g^{-1}g)=(\widehat{Ad_gX}F)(g),\]
and since $P_t$ commutes with right-invariant vector fields, it follows that
\[(XP_tf)(g)=(\widehat{Ad_gX}P_tf)(g)=(P_t\widehat{Ad_gX}f)(g).\]
Using Corollary \ref{c.1}, we see that
\[|(\widehat{Ad_gX}f)(h)|\leq |\widehat{Ad_gX}(h)||\nabla f(h)|\leq C^2e^{c(|g|_H+|h|_H)}|\nabla f(h)|,\]
and it follows that
\begin{align*}
|(P_t\widehat{Ad_gX}f)(g)|&
\leq \int_G |\widehat{Ad_gX}(h)||\nabla f(h)|p_t(h^{-1}g)~d\mu(h)\\ &
\leq C^2\int_G |\nabla f(h)|e^{c|h|_H}e^{c|g|_H}p_t(h^{-1}g)~d\mu(h)\\ & \leq C^2\int_G |\nabla f(h)|e^{2c|h|_H}e^{c|h^{-1}g|_H}p_t(h^{-1}g)~d\mu(h)
\end{align*}
where in the third inequality we've used the fact that $|g|_H=|hh^{-1}g|_H\leq |h|_H+|h^{-1}g|_H$.  Now $|\nabla f|$ is supported on a compact set $\Omega\subset G$.  So in fact
\begin{align*}
|(P_t\widehat{Ad_gX}f)(g)|&\leq \tilde{C}\int_Ge^{c|h^{-1}g|_H}p_t(h^{-1}g)~d\mu(h)\\ &=\tilde{C}\int_G e^{c|u^{-1}|_H}p_t(u^{-1})~d\mu(u),
\end{align*}
where $\tilde{C}=C^2\left(\sup_{h\in \Omega}|\nabla f(h)|e^{2c|h|_H}\right)$ and the second line follows from the change of variables $u=g^{-1}h$.  It follows that
\[\sup_{0< t\leq T}||XP_tf||_\infty\leq \tilde{C}\left(\sup_{0< t\leq T} \int_G e^{c|u^{-1}|_H}p_t(u^{-1})~d\mu(u)\right).\]
Now $|u^{-1}|_H=|u|_H$ and $p_t(u^{-1})~d\mu(u)=m^{-1}(u)p_t(u)~d\mu(u)$, where $m$ denotes the modular function.  Since $m$ is a homomorphism, by arguments similar to those of Proposition \ref{p.4.4.1}, one can show that $m$ (and hence $m^{-1}$) has at most exponential growth in the horizontal distance.  Therefore, to prove the desired result, it suffices to show that for any $k>0$, we have
\begin{equation}\label{e.9292}
\sup_{0< t\leq T} \int_G e^{k|u|_H}p_t(u)~d\mu(u)<\infty.
\end{equation}
That the distance is exponentially integrable with respect to the heat kernel in this setting is discussed in Section IV.4b of \cite{robinson}.  We will provide an alternative proof.

We first observe that it is sufficient to prove Eq. (\ref{e.9292}) for small $T$, since the semigroup property and the fact that $u\rightarrow e^{k|u|_H}$ is sub-multiplicative will then give the result for larger $T$.  To see this, suppose Eq. (\ref{e.9292}) is satisfied and suppose $T<S\leq 2T$.  Set $t=S-T$. Then
\begin{align*}
\int_G e^{k|u|_H}p_S(u)d\mu(u) & =\int_G e^{k|u|_H}\left(\int_G p_t(x)p_T(x^{-1}u)~d\mu(x)\right)~d\mu(u)\\ & =\int_G p_t(x)\left(\int_G e^{k|u|_H}p_T(x^{-1}u)~d\mu(u)\right)~d\mu(x)\\ & =\int_G p_t(x)\left(\int_G e^{k|xy|_H}p_T(y)~d\mu(y)\right)~d\mu(x)\\ &\leq \int_G e^{k|x|_H}p_t(x)\left(\int_G e^{k|y|_H}p_T(y)~d\mu(y)\right)~d\mu(x)\\ &\leq \left(\sup_{0< t\leq T} \int_G e^{k|u|_H}p_t(u)~d\mu(u)\right)^2
\end{align*}
where in the third line above we have used the change of variables $y=x^{-1}u$.  More generally,
\[\sup_{0< t\leq nT} \int_G e^{k|u|_H}p_t(u)~d\mu(u)\leq \left(\sup_{0< t\leq T} \int_G e^{k|u|_H}p_t(u)~d\mu(u)\right)^n.\]

Now using the heat kernel upper bounds of Eq. (\ref{e.4.4.4}) with $\epsilon=\frac{1}{2}$ and the heat kernel lower bounds of Eq. (\ref{e.4.4.5}), we note that for any $t>0$ and $0<S<1$,
\begin{align}
\int_G e^{k|u|_H}p_t(u)~d\mu(u) &=\int_G e^{k|u|_H}\frac{p_t(u)}{p_S(u)}p_S(u)~d\mu(u)\nonumber\\ &\leq \int_G e^{k|u|_H}\frac{Ct^{-\nu/2}e^{at}e^{-2|u|_H^2/9t}}{CS^{-\nu/2}e^{-C|u|_H^2/S}}p_S(u)~d\mu(u)\nonumber\\&\leq \left(\frac{S}{t}\right)^{\nu/2}e^{at}\int_G e^{k|u|_H+|u|_H^2\left(-\frac{2}{9t}+\frac{C}{S}\right)}p_S(u)~d\mu(u)\nonumber
\end{align}
Now fix $0<t\leq \frac{1}{10C}$ where $C$ is the constant in Eqs. (\ref{e.4.4.4}) and (\ref{e.4.4.5}), and set $S:=9tC<1$.  For these values of $t$ and $S$, the above calculations imply
\begin{align*}
\int_G e^{k|u|_H}p_t(u)~d\mu(u) &\leq \left(9C\right)^{\nu/2}e^{at}\int_G e^{k|u|_H+|u|_H^2\left(-\frac{2}{9t}+\frac{1}{9t}\right)}p_S(u)~d\mu(u)\\ &=\left(9C\right)^{\nu/2}e^{at}\int_G e^{k|u|_H+|u|_H^2\left(-\frac{1}{9t}\right)}p_S(u)~d\mu(u)\\ &=\left(9C\right)^{\nu/2}e^{at}\int_G e^{-(\frac{1}{3\sqrt{t}}|u|_H-\frac{3k\sqrt{t}}{2})^2}e^{\frac{9k^2t}{4}}p_S(u)~d\mu(u)\\ &\leq \left(9C\right)^{\nu/2}e^{at}e^{\frac{9k^2t}{4}}.
\end{align*}
It follows that
\[\sup_{0< t\leq \frac{1}{10C}} \int_G e^{k|u|_H}p_t(u)d\mu(u)<\infty,\]
which establishes the desired result.
\end{proof}

Theorem \ref{t.4.4.1} shows that Hypothesis 1.4 of \cite{BG1} is satisfied in this general setting.  This allows us to use the following parabolic comparison in the next subsection, where we will apply the following Proposition with $u$ and $v$ equal to linear combinations of $\Gamma(P_tf)$ and $\Gamma^R(P_tf)$.  A more general statement and proof can be found as Proposition 4.5 of \cite{BG1}.
\begin{proposition}\label{p.4.4.3}
Let $T>0$.  Suppose that $u,v:G\times[0,T]\rightarrow \mathbb{R}$ are smooth functions such $\sup_{t\in [0,T]}||u(\cdot,t)||_\infty<\infty$ and $\sup_{t\in [0,T]}||v(\cdot,t)||_\infty<\infty$.  Suppose
\[Lu+\frac{\partial u}{\partial t}\geq v\]
on $G\times [0,T]$.  Then for all $x\in G$,
\[P_T(u(\cdot,T))(x)\geq u(x,0)+\int_0^TP_s(v(\cdot,s))(x)ds.\]
\end{proposition}

\subsubsection{Functional Inequalities}\label{s.38}

We now return to the specific case examined in this work; namely, $G$ is a 3-dimensional solvable Lie group with Lie algebra $\mathfrak{g}$.  In section \ref{s.4.3}, it was shown that for any $f\in C^\infty(G)$,
\begin{equation}\label{e.4.4.2.1}
\Gamma_2(f)+\nu\Gamma_2^R(f)\geq \frac{1}{2}(Lf)^2+\frac{1}{2}(1-\nu^2\alpha^2)\Gamma^R(f)+(-\alpha^+-\beta^2-\frac{1}{\nu})\Gamma(f).
\end{equation}
Our results in the section concern the case where $\alpha\neq 0$.  Recall that by Remark \ref{r.4.3.1}, the case where $\alpha=0$ is covered in previous work \cite{BB2}.  Note that Theorem \ref{t.4.4.1} indicates that, for any $T\geq 0$, $\sup_{0\leq t\leq T}||\Gamma(P_tf)||_\infty<\infty$ and $\sup_{0\leq t\leq T}||\Gamma^R(P_tf)||_\infty<\infty$ whenever $f\in C_c^\infty(G)$.

\begin{proposition}\label{p.4.4.5}
Suppose $G$ has parameters $(\alpha,\beta)$ with $\alpha\neq 0$ and set $\kappa=\beta^2+\alpha^+$.  Then for every $f\in C_c^\infty$ and $T\in [0,\frac{1}{2\kappa}\ln{\left(\frac{\kappa+|\alpha|}{|\alpha|}\right)})$,
\begin{equation}\label{e.3102}
\Gamma(P_Tf)+\frac{1}{|\alpha|}\Gamma^R(P_Tf)\leq \frac{\kappa e^{2\kappa T}}{\kappa+|\alpha|(1-e^{2\kappa T})}P_T(\Gamma(f))+\frac{1}{|\alpha|}P_T(\Gamma^R(f)).
\end{equation}
\end{proposition}
\begin{proof}
Let $f\in C_c^\infty$, and let $T$ be a positive number to be determined later in the proof.  For $0\leq t\leq T$ and $x\in G$, set
\[\phi_1(x,t)=\Gamma(P_{T-t}f)(x)\]
and
\[ \phi_2(x,t)=\Gamma^R(P_{T-t}f)(x).\]
An easy computation shows that
\begin{align*}
L\phi_1+\frac{\partial\phi_1}{\partial t}&=2\Gamma_2(P_{T-t}f)
\end{align*}
and
\[L\phi_2+\frac{\partial\phi_2}{\partial t}=2\Gamma^R_2(P_{T-t}f).\]
Now set
\[\phi(x,t)=a(t)\phi_1(x,t)+b(t)\phi_2(x,t),\]
where $a(t)$ and $b(t)$ are positive differentiable functions.  Then using the curvature-dimension inequality Eq. (\ref{e.4.4.2.1}),
\begin{align*}
L\phi+\frac{\partial\phi}{\partial t} &= a'\phi_1(x,t)+b'\phi_2(x,t)+2a\Gamma_2(P_{T-t}f)+2b\Gamma_2^R(P_{T-t}f)\\ &=a'\phi_1(x,t)+b'\phi_2(x,t)+2a\left(\Gamma_2(P_{T-t}f)+\frac{b}{a}\Gamma_2^R(P_{T-t}f)\right)\\ &\geq a'\phi_1(x,t)+b'\phi_2(x,t)\\&+2a\left(\frac{1}{2}(L(P_{T-t}f))^2+\frac{1}{2}(1-\frac{b^2}{a^2}\alpha^2)\Gamma^R(P_{T-t}f)+(-\alpha^+-\beta^2-\frac{a}{b})\Gamma(P_{T-t}f)\right)\\& =\left(a'-2a(\alpha^++\beta^2)-\frac{2a^2}{b}\right)\phi_1(x,t)+\left(b'+a-\frac{\alpha^2b^2}{a}\right)\phi_2(x,t)\\&+a(L(P_{T-t}f))^2
\end{align*}
We now look for two functions $a(t)$ and $b(t)$ which solve
\begin{equation}\label{e.3001}
\begin{array}{ccc} a'(t)-2\kappa a(t)-2\frac{a(t)^2}{b(t)}&\geq & 0\\b'(t)+a(t)-\alpha^2\frac{b(t)^2}{a(t)}&\geq &0\end{array}
\end{equation}
where $\kappa:=\alpha^++\beta^2\geq 0$.  Note that the first equation above necessarily implies that $a'(t)\geq 0$.  Set $A=a(0)>0$.  Then since $a(t)\geq A$ for all $t$,
\[b'(t)+a(t)-\alpha^2\frac{b(t)^2}{a(t)}\geq b'(t)+A-\alpha^2\frac{b(t)^2}{A},\]
and so our candidate for $b(t)$ is a solution to the equation
\[b'(t)=-A+\alpha^2\frac{b(t)^2}{A}.\]
This is an autonomous equation with a constant solution $b(t)=\frac{A}{|\alpha|}$.  For this choice of $b(t)$, we find $a(t)$ by solving the equation
\[a'(t)=2\kappa a(t)+2\frac{|\alpha| a(t)^2}{A}\qquad a(0)=A,\]
which yields the solution
\begin{equation}\label{e.3100}
a(t)=\frac{A\kappa e^{2\kappa t}}{\kappa+|\alpha|(1-e^{2\kappa t})}.
\end{equation}
This function $a(t)$ is defined on the interval $[0,\frac{1}{2\kappa}\ln{\left(\frac{\kappa+|\alpha|}{|\alpha|}\right)})$.

With this choice of $a(t)$ and $b(t)$, when $0\leq t\leq T<\frac{1}{2\kappa}\ln{\left(\frac{\kappa+|\alpha|}{|\alpha|}\right)}$,
\begin{equation}\label{e.3101}
L\phi+\frac{\partial\phi}{\partial t}\geq a(L(P_{T-t}f))^2\geq 0
\end{equation}
It follows by the parabolic comparison theorem (Proposition \ref{p.4.4.3}) that
\[P_{T}(\phi(x,T))\geq \phi(x,0).\]
Now
\begin{align*}
\phi(x,0)&=a(0)\Gamma(P_Tf)(x)+b(0)\Gamma^R(P_Tf)\\&=A\Gamma(P_Tf)(x)+\frac{A}{|\alpha|}\Gamma^R(P_Tf)(x)
\end{align*}
while
\begin{align*}
P_T(\phi(x,T))&=a(T)P_T(\Gamma(f))+b(T)P_T(\Gamma^R(f))\\ &=\frac{A\kappa e^{2\kappa T}}{\kappa+|\alpha|(1-e^{2\kappa T})}P_T(\Gamma(f))+\frac{A}{|\alpha|}P_T(f\Gamma^R(f)).
\end{align*}
This yields Eq. (\ref{e.3102}).
\end{proof}

\begin{proposition}\label{p.4.4.6}
Suppose $G$ has parameters $(\alpha,\beta)$ with $\alpha\neq 0$ and set $\kappa=\beta^2+\alpha^+$.  Then for every $f\in C_c^\infty$ and $T>0$,
\begin{align*}
&\Gamma(P_T f)+\frac{1}{|\alpha|}\Gamma^R(P_Tf)-e^{2(\kappa+|\alpha|)T}P_T\Gamma(f)-\frac{1}{|\alpha|}P_T\Gamma^R(f)\\ &\leq |\alpha|e^{2(\kappa+|\alpha|)T}(e^{2(\kappa+|\alpha|)T}-1)\left(P_Tf^2-(P_Tf)^2\right).
\end{align*}
\end{proposition}
\begin{proof}
Let $T>0$ and define $\phi$ as in the proof of Proposition (\ref{p.4.4.5}).  Then by the same computations as above,
\begin{equation}\label{e.4000}
L\phi+\frac{\partial\phi}{\partial t}\geq \left(a'-2a\kappa-\frac{2a^2}{b}\right)\Gamma(P_{T-t}f)+\left(b'+a-\frac{\alpha^2b^2}{a}\right)\Gamma^R(P_{T-t}f)
\end{equation}
Assume that $a(t)$ is increasing with $a(0)=A$, in which case $b(t)=\frac{A}{|\alpha|}$ satisfies
\[b'+a-\frac{\alpha^2b^2}{a}\geq 0\]
This was shown in the proof of Proposition (\ref{p.4.4.5}) above.  For this choice of $b(t)$, then Eq. (\ref{e.4000}) becomes
\[L\phi+\frac{\partial \phi}{\partial t}\geq \left(a'-2a\kappa-\frac{2|\alpha|a^2}{A}\right)\Gamma(P_{T-t}f).\]
Now set
\[a(t)=Ae^{2(\kappa+|\alpha|)t}.\]
Then notice that
\[a'-2a\kappa-\frac{2|\alpha|a^2}{A} = 2|\alpha|Ae^{2(\kappa+|\alpha|)T}(1-e^{2(\kappa+|\alpha|)T})<0.\]
It follows that for this choice of $a(t)$, we now have
\[L\phi+\frac{\partial \phi}{\partial t}\geq 2|\alpha|Ae^{2(\kappa+|\alpha|)T}(1-e^{2(\kappa+|\alpha|)T})\Gamma(P_{T-t}f).\]
Now using the parabolic comparison theorem (Proposition \ref{p.4.4.3}) we have
\[P_{T}(\phi(x,T))\geq \phi(x,0)+2|\alpha|Ae^{2(\kappa+|\alpha|)T}(1-e^{2(\kappa+|\alpha|)T})\int_0^T P_t\Gamma(P_{T-t}f)~dt.\]
This gives the result after noticing that
\[P_t\Gamma(P_{T-t}f)=\frac{1}{2}\frac{d}{dt}P_t(P_{T-t}f)^2.\]

\end{proof}

\end{document}